\documentclass[12pt]{amsart}
\usepackage{amssymb,latexsym, amsmath, amscd, array, graphicx}

\swapnumbers \numberwithin{equation}{section}

\theoremstyle{plain}

\newtheorem{thm}{Theorem}[section]
\newtheorem{theorem}[thm]{Theorem}
\newtheorem{lem}[thm]{Lemma}

\newtheorem{prop}[thm]{Proposition}
\newtheorem{cor}[thm]{Corollary}

\theoremstyle{definition}
\newtheorem{defin}[thm]{Definition}

\newtheorem{remark}[thm]{Remark}

\newtheorem{ex}[thm]{Example}

\newtheorem{question}[thm]{Question}

\DeclareMathOperator{\Con}{{\rm Con}}
\DeclareMathOperator{\cat}{{\mbox{\rm cat$_{\rm LS}$}}}

\DeclareMathOperator{\cuplength}{{\rm cup-length}}

\DeclareMathOperator{\Int}{{\rm Int}}

\def\Int{\protect\operatorname{Int}}

\def\cat{\protect\operatorname{cat}}
\def\Cat{\protect\operatorname{Cat}}

\def\scr{\mathcal}

\def\C{{\mathbb C}}
\def\Z{{\mathbb Z}}

\def\R{{\mathbb R}}
\def\N{{\mathbb N}}

\def\1{\hbox{\rm\rlap {1}\hskip.03in{\rom I}}}
\def\Bbbone{{\rm1\mathchoice{\kern-0.25em}{\kern-0.25em}
{\kern-0.2em}{\kern-0.2em}I}}

\long\def\forget#1\forgotten{} %

\newcommand\ver[1]{\marginpar{\tiny Changed in Ver \VER}}

\date{\today}
\begin{document}

\title[On topological complexity and LS-category]{On topological complexity and LS-category }

\author[A.~Dranishnikov]{Alexander  Dranishnikov} %
\thanks{Supported by NSF, grant DMS-0904278}

\address{Alexander N. Dranishnikov, Department of Mathematics, University
of Florida, 358 Little Hall, Gainesville, FL 32611-8105, USA}
\email{dranish@math.ufl.edu}

\subjclass[2000]
{Primary 55M30; 
Secondary 57N65, 54F45}  

\begin{abstract} We present some results supporting the Iwase-Sakai conjecture about coincidence of the topological 
complexity $TC(X)$ and monoidal topological complexity $TC^M(X)$. Using these results
we provide  lower and  upper bounds for the topological complexity of the wedge $X\vee Y$.  We use these bounds to give a counterexample to the conjecture asserting that 
$TC(X')\le TC(X)$ for any covering map $p:X'\to X$. 

We discuss a possible reduction of the monoidal topological complexity to the LS-category.
Also we apply the LS-category to give a short proof of the Arnold-Kuiper theorem.
\end{abstract}

\maketitle

\section{Introduction}

Let $PX=X^{[0,1]}$ denote the  space of all paths in $X$. Let $i_X:X\to PX$ be the inclusion of $X$ into $PX$ as a subspace of constant paths.
There is a natural fibration $\pi:PX\to X\times X$ defined as $\pi(f)=(f(0),f(1))$ for $f\in PX$, $f:[0,1]\to X$.

Let $X$ be an ENR. A section $s:X\times X\to PX$ of $\pi$ is called a {\em motion planning algorithm}. We say that a motion planning algorithm $s$ has  {\em complexity} $k$
if $X\times X$ can be presented as a disjoint union $F_1\cup\dots\cup F_k$ of ENRs such that
$s$ is continuous on each $F_i$. The {\em topological complexity} $TC(X)$ of a space $X$ was defined by Farber as the minimum of $k$ such that there is a motion planning algorithm
of complexity $k$~\cite{F1}. Equivalently, $TC(X)$ is the minimal number $k$
such that $X\times X$ admits an open cover $U_1,\dots, U_k$ such that over each $U_i$ there is a continuous section of $\pi$.

We say that a motion planning algorithm $s:X\times X\to PX$ is {\em reserved} if $s|_{\Delta X}=i_X$ where $\Delta X\subset X\times X$ is the diagonal. In other words, if the initial position of a robot in the configuration space $X$
coincides with the terminal position, then the algorithm keeps the robot still. This condition on the motion planning algorithms seems to be very natural. The corresponding complexity of a space $X$ was denoted by  Iwase and Sakai as $TC^M(X)$ and was called
the {\em monoidal topological complexity} of $X$~\cite{IS1}. In the original definition they
additionally assumed that all sets $U_i$ contain the diagonal. Their definition agrees with the above since their condition always can be achieved by reduction of
an open cover $U_1,\dots, U_k$ with reserved sections $s_i$ to a closed cover $F_1,\dots, F_k$, $F_i\subset U_i$, then by
adding the diagonal to each $F_i$ with the natural extension of the sections $\bar s_i$,  and then by taking open enlargement $V_i$ of the sets $F_i\cup\Delta X$ that admit extensions of the sections $\bar s_i$.

Iwase and Saki conjectured that $TC^M(X)=TC(X)$. 
In fact, first they  gave a proof to the conjecture in \cite{IS1} and then withdrew it in \cite{IS2}. We prove this conjecture under the assumption
$TC(X)>\dim X+1$.
Also, using the Weinberger Lemma from~\cite{F3} we show that the conjecture holds true when $X$ is a Lie group.

The topological complexity is closely related to the Lusternik-Schnirelmann category $\cat(X)$ of a space  which is defined as the minimal number $k$ such that $X$ can be covered by $k$ open sets $U_1,\dots, U_k$ all contractible to a point in $X$.
We denote by $$\Cat(X)= \cat(X)-1,$$ the reduced LS-category. The reduced category appears naturally in several inequalities in the theory~\cite{CLOT}:
$$
\cuplength(X)\le \Cat(X)\le \dim (X)
$$
and

$$
\Cat(X\times Y)\le \Cat(X)+\Cat(Y).$$
In the first inequality the cup-length is taken for any reduced cohomology (possibly twisted).

Some of the formulas for $\cat$ translate to similar statements for $TC$. For example for $TC$ there is an inequality similar to the above  for the product of two spaces~\cite{F4}. Also there are analogous estimates of TC in terms of the cup product and dimension~\cite{F4}.
On the other hand, the simple $\cat$ formula for the  wedge $\cat(X\vee Y)=\max\{\cat X, \cat Y\}$ does not hold for $TC$. So far there is no  nice analog of it for TC.
The best that we can prove here is Theorem~\ref{wedge} from this paper. Another example is the formula $\cat(Y)\ge \cat (X)$ for a covering map $p:X\to Y$ which supports an intuitive idea that a covering space is always simpler than the base. So it was natural to assume that the same holds true for TC.  I've learned about this problem from Yuli Rudyak. In this paper Theorem~\ref{example} gives a negative answer  to this question.

There have been several attempts to reformulate the  topological complexity of $X$ as some modified category of a related space.
In this paper we discuss a possible characterization of the monoidal topological complexity in terms of the category.  We define a $rel\ \infty$ category $\infty\text{-}\cat(Y)$ of non-compact spaces $Y$ and
discuss the problem of coincidence between $\cat(X/A)$ and $\infty\text{-}\cat(X\setminus A)$ for a subcomplex $A\subset X$ of a finite complex $X$.
Then we show that $TC^M(X)$ is always between $\cat(X\times X)/\Delta(X)$ and $\infty\text{-}\cat(X\times X\setminus\Delta X)$.

Note that both $\cat(X)$ and $TC(X)$ are partial case of the Schwarz genus~\cite{Sch}: $\cat(X)=sg(\pi_0:P_0X\to X)$ and $TC(X)=sg(\pi:PX\to X\times X)$
where $P_0X\subset PX$ is the subspace of paths $f:[0,1]\to X$ that start in a base point $x_0\in X$, $f(0)=x_0$, and $\pi_0(f)=f(1)$.
We recall the {\em Schwarz genus}~\cite{Sch} of a fibration $p:X\to Y$ is the minimal number of open sets $U_1,\dots, U_k$ that cover $Y$ and admit sections $s_i:U_i\to X$ of $p$. In the paper we estimate the  Schwarz genus~\cite{Sch} of arbitrary fibration $p:X\to Y$
in terms the category of its mapping cone $C_p$.

Finally, we apply the LS-category to give a short proof of the Arnold-Kuiper theorem which states that the orbit space of the action of $\Z_2$ on
the complex projective plane $\C P^2$ by the conjugation is the 4-sphere. Note that this theorem was discovered by Arnold~\cite{Ar1} who published his proof  much later~\cite{Ar2}. It was proven independently by Kuiper~\cite{K} and by Massey~\cite{M}. 

The author is thankful to Michael Farber and Yuli Rudyak for helpful conversations and to Peter Landweber for valuable remarks.

\section{Monoidal topological complexity}

\begin{thm}  For ENR spaces,
$$TC(X)\le TC^M(X)\le TC(X)+1.$$
\end{thm}
This theorem was proved in~\cite{IS2}. Since the proof there is too technical we give an alternative proof.
\begin{proof}
The first inequality is obvious. Since $X$ is ANR, there  is an open neighborhood $W$ of the diagonal $\Delta X$ in $X\times X$ and a continuous map
$\phi:W\times [0,1]\to X$ such that $\phi(x,x',0)=x$, $\phi(x,x',1)=x'$, and $\phi(x,x,t)=x$ for all $t\in[0,1]$, $x,x'\in X$.
Let $U_1,\dots, U_n$ be an open cover of $X\times X$ by sets that admit sections $s_i:U_i\to PX$ of $\pi$. Let $F$ be a closed neighborhood of $\Delta X$
that lies in $W$. Then all sets in the open cover  $U_1\setminus F,\dots,U_n\setminus F, W$ of $X\times X$ admit reserved sections. Hence
$TC^M(X)\le n+1$.
\end{proof}

Note that the path fibration $\pi:PX\to X\times X$ restricted over the diagonal defines the free loop fibration $p:LX\to X$.
A {\em canonical} section $\bar s:\Delta X\to LX$ of $p$ is defined as $\bar s(x)=c_x$, where $c_x:I\to X$ is the constant map to $x$.

We use the standard convention to denote the elements of the iterated join product
$X_1\ast X_2\ast\dots\ast X_n$   as formal linear combinations $t_1x_1+t_2x_2+\dots+t_nx_n$,  $\sum t_i=1$, $t_i\ge 0$,
$x_i\in X_i$ where all summands  of the type $0x_i$  are dropped.
We use the notation $\ast^nX$ for the iterated join product of $n$ copies of
$X$ with itself.

We recall that a fiber-wise join of maps $f_i:X_i\to Y$, $i=1,\dots,n$ is the map $$f_1\tilde\ast \cdots\tilde\ast f_n:X_1\tilde\ast_Y \cdots\tilde\ast_Y X_n\to Y$$ where 
$$
X_1\tilde\ast_Y\cdots\tilde\ast_YX_n=\{t_1x_1+\dots+t_nx_n\in X_1\ast \dots\ast X_n\mid f_1(x_1)=\dots=f_n(x_2)\}
$$
is the fiber-wise join of spaces $X_1,\dots, X_n$
and $$(f_1\tilde\ast \dots\tilde\ast f_n)(t_1x_1+\dots+t_nx_n)=f_i(x_i).$$
Thus, the preimage $(f_1\tilde\ast\cdots\tilde\ast f_n)^{-1}(y)$ of a point $y\in Y$
is the join product of the preimages
$f_1^{-1}(y)\ast\dots\ast f_n^{-1}(y)$.

We define $P_nX=PX\tilde\ast_{X\times X}\cdots\tilde\ast_{X\times X}PX$ and $$\pi_n=\pi\tilde\ast\cdots\tilde\ast\pi:P_nX\to X\times X$$ 
to be the fiber-wise join product of $n$ copies of $\pi$. 
Note that there are  imbeddings $P_1X\subset P_2X\subset\dots\subset P_nX$ such that $\pi_i|_{P_{i-1}}=\pi_{i-1}$.
Then the section $\bar s:X\times X\to P_1X$ of $\pi_1$ can be regarded as a section of $\pi_n$. Also we define $p_1=p:LX\to X$, $L_nX=L_{n-1}\tilde\ast_XLX$, and
 $p_n=p_{n-1}\tilde\ast p:L_n\to X$. Note that $\pi_n^{-1}(\Delta X)\cong L_nX$ and $p_n$ is the restriction of $\pi_n$ to $\pi_n^{-1}(\Delta X)$.
Note also that the canonical section $\bar s$ defines a trivial subbundle $p'_n:E\to X$ of $p_n$ with the fiber the $(n-1)$-simplex $\Delta^{n-1}$.

We recall that a map $p:E\to B$  satisfies the {\em Homotopy Lifting Property for a pair} $(X,A)$ if for any
homotopy $H:X\times I\to B$ with a lift $H':A\times I\to E$ of the restriction $H|_{A\times I}$
and a lift $H_0$ of $H|_{X\times0}$ which agrees with $H'$, there is a lift $\bar H:X\times I\to E$
of $H$ which agrees with $H_0$ and $H'$.
The following is well-known~\cite{H}:
\begin{thm}\label{lift}
Any Hurwicz fibration $p:E\to B$  satisfies the Homotopy Lifting Property for CW complex pairs $(X,A)$.
\end{thm}
\begin{cor}\label{HLP}
Let $p:E\to X$ be a Hurewicz fibration with a section $s:X\to E$. A fiber-wise homotopy $G:A\times I\to E$ of the restriction $s|_A$ to a closed subset $A\subset X$ can be extended to a fiber-wise homotopy $\bar G:X\to E$ of $s$ provided $(X,A)$ is a CW complex pair.
\end{cor}

\begin{prop}\label{ganea} For CW complexes $X$,

(1) $TC(M)\le n\ \Leftrightarrow\ \  \pi_n:P_nX\to X\times X$ admits a section.

(2) $TC^M(M)\le n\ \Leftrightarrow\ \pi_n:P_nX\to X\times X$ admits a section $s$ which agrees with the canonical section over the diagonal
$s|_{\Delta X}=\bar s$.
\end{prop}
\begin{proof}
The statement (1) is a part of a general theorem proven by Schwartz~\cite{Sch} for fibrations $q:X\to Y$:
$sg(q)\le n$ if and only if the $n$-fold iterated fiber-wise join product $\tilde\ast^nq:\tilde\ast^n_YX\to Y$ admits a section.

The implication $\Leftarrow$ in (2) is obvious. For the other direction we note that $n$ reserved sections $s_i:U_i\to PX$ defined for an open cover
$U_1,\dots, U_n$ of $X\times X$ define a section $s$ of $\pi_n$ with the image $s(\Delta X)$ lying in $E$. Therefore over $\Delta X$ it could be fiber-wise deformed to $\bar s$. By Proposition~\ref{lift} that deformation
can be extended to a fiber-wise deformation of $s$.
\end{proof}

\begin{thm}\label{M} The equality
$$TC(X)= TC^M(X)$$
holds true 
for $k$-connected simplicial complexes $X$ such that $$(k+1)TC(X)>\dim(X)+1.$$
\end{thm}
\begin{proof}
Let $TC(X)=n$. Note that the fiber $\pi^{-1}(x,x')$ is homotopy equivalent to the loop space $\Omega(X)$. Since $\Omega(X)$ is $(k-1)$-connected, the iterated join product $\ast^n\Omega(X)$ is $((k+1)n-2)$-connected.
We show that any section $s:\Delta X\to L_nX$ can be fiber-wise joined by a homotopy with a canonical section $\bar s:\Delta X\to L_nX$.
By induction on $i$ we construct a section $s_i:X\to L_nX$, that coincides with $\bar s$ on the $i$-skeleton $X^{(i)}$, together with a fiber-wise
homotopy joining $s$ and $s_i$.  Here we use the identification $\Delta X= X$. For $i=0$ we take paths in the fibers $p_n^{-1}(v)$ joining $s(v)$ and $\bar s(v)$ for all $v\in X^{(0)}$.
Then we extend them to a fiber-wise homotopy of $s$ to a section $s_0$. Assume that $s_{i-1}$ is already constructed and $i\le\dim X\le (k+1)n-2$. 
Independently for every $i$-simplex 
$\sigma\subset X$ we consider the problem of joining $s_{i-1}$ with $\bar s$ over $\sigma$ by a fiber-wise homotopy. Since the fiber bundle $p_n$ is trivial over $\sigma$
with a $i$-connected fiber, the identity homotopy on the boundary $\partial\sigma$ can be extended to a homotopy between $\bar s|_{\sigma}$ and $s_{i-1}|_{\sigma}$. This extension can be deformed to a fiber-wise  homotopy. All these homotopies together define a fiber-wise  homotopy between $s_{i-1}$ and $\bar s$
over $X^{(i)}$. Since $(X,X^{(i)})$ is a CW pair, by Proposition~\ref{lift} we can extend it to a fiber-wise homotopy over $X$.

Let $s:X\times X\to P_nX$ be a section. On $\Delta X$ it can be deformed to a canonical section $\bar s$. Since $(X\times X,\Delta X)$ is a CW pair, by Proposition~\ref{lift} there is a fiber-wise homotopy of $s$ to a section $s'$ that coincides with $\bar s$ on $\Delta X$. Therefore, $TC^M(X)\le n$.
\end{proof}

\begin{cor}
$TC(S^m)=TC^M(S^m)$ for all $m$.
\end{cor}

The following is an extension of Weinberger's Lemma from~\cite{F3} to the case of monoidal topological complexity.
\begin{lem}\label{weinberger} 
For a connected Lie group $G$,
$$
TC(G)=TC^M(G)=\cat (G).$$
\end{lem}
\begin{proof} In view of what is already known~\cite{F3}, it suffices to show the inequality $TC^M(G)\le \cat(G)$.
Let $\cat(G)=n$ and let $U_1,\dots, U_n$ be an open cover of $G$ together with homotopies $H_i:U_i\times[0,1]\to G$ contracting $U_i$ to the unit $e\in G$.
Clearly, we may assume that $e\notin U_i$ for $i>1$. Since the inclusion $e\in G$ is a cofibration, we may assume that $H_1(e,t)=e$ for all $t$. Then for the open cover of $G\times G$ as defined in~\cite{F3}
$$W_i=\{(a,b)\in G\times G\mid a^{-1}b\in U_i\}$$ the sections $s_i:W_i\to PG$ defined as
$$
s_i(a,b)(t)=ah_i(a^{-1}b,t)\in G,\ \ \ (a,b)\in W_i
$$
are reserved.
Indeed, $\Delta G\cap W_i=\emptyset$ for $i>1$ and $$s_1(a,a)(t)=ah_1(a^{-1}a,t)=ah_1(e,t)=ae=a$$ for all $(a,a)\in \Delta G$.
\end{proof}

\section{Topological complexity of wedge and covering maps}

A {\em deformation} of $U\subset Z$ in $Z$ to a subset $A\subset Z$ is a continuous map 
$D:U\times I\to Z$ such that: $D(u,0)=u$, $D(u,1)\in A$ for all $u\in U$.
A {\em strict deformation} of $U\subset Z$ in $Z$ to $A\subset Z$ is a deformation
$D:U\times I\to Z$ such that
$D(u,t)=u$ for all $t\in I$ whenever $u\in A$. 

\begin{prop}\label{sections} 
Let $X$ be a metric space.
For an open set $U\subset X\times X$ the following are equivalent:

(1) There is a reserved section $s:U\to PX$ over $U$ of the fibration $\pi:PX\to X\times X$.

(2) There is a strict deformation $D:U\times I\to X\times X$ to the diagonal $\Delta X=\{(x,x)\in X\times X\mid x\in X\}$ 

(3) For any choice of a base point $x_0\in X$ there is a strict deformation $D$ of $U$ to $\Delta X$ which preserves faces
$X\times x_0$ and $x_0\times X$, i.e.,
for all  $t\in I$,
$$D((x,x_0),t)\in X\times x_0\ \ \text{and}\ \ \ D((x_0,x),t)\in x_0\times X.$$ 
\end{prop}
\begin{proof} (1) $\Rightarrow$ (3). Let $\|x\|=d(x,x_0)$. We define
$$
D((x,y),t)=(s(x,y)(\frac{\|x\|}{\|x\|+\|y\|}t),s(x,y)(1-\frac{\|y\|}{\|x\|+\|y\|}t)
$$
if $(x,y)\ne (x_0,x_0)$ and define $D((x_0,x_0),t)=(x_0,x_0)$. Since $s(x,y)(0)=x$ and $s(x,y)(1)=y$, we obtain that $D((x,y),0)=(x,y)$.
Note that $$D((x,y),1)=(s(x,y)(\frac{\|x\|}{\|x\|+\|y\|}),s(x,y)(\frac{\|x\|}{\|x\|+\|y\|})\in\Delta X.$$ Since the section $s$ is reserved, $D((x,x),t)=(s(x,x)(t/2),s(x,x)(t/2))=(x,x)$. Note that $$D((x,x_0),t)=(s(x,x_0)(t),s(x,x_0)(1))=(s(x,x_0)(t),x_0)\in X\times x_0$$ and
$$D((x_0,y),t)=(s(x_0,y)(0),s(x_0,y)(1-t))=(x_0,s(x_0,y)(1-t))\in x_0\times X.$$
The deformation $D$ is continuous at $(x_0,x_0)$ ( if defined) since the section $s(x_0,x_0)$ is stationary at $(x_0,x_0)$.

(3) $\Rightarrow$ (2) is obvious.

(2) $\Rightarrow$ (1). Let $pr_1:X\times X\to X$ denote the projection to the first factor and
$pr_2:X\times X\to X$ to the second. Given a strict deformation $D$ we define a section $s:U\times I\to PX$ as follows:
$$
s(x,y)(t)=\begin{cases}pr_1D((x,y),2t) & \ \text{if}\ \  t\le 1/2\\
                         pr_2D((x,y),2-2t) & \ \text{if}\ \ t\ge 1/2.\\
\end{cases}
$$
This path is well-defined since $D((x,y),1)\in\Delta X$. Clearly it is a path from $x$ to $y$. If $x=y$, the path is stationary.
Thus $s$ is a reserved section.
\end{proof}

\begin{prop}\label{retraction}
Let $A$ be a retract of an ENR space $X$. Then $TC(X)\ge TC(A)$.
\end{prop}
\begin{proof} Let $r:X\to A$ be a retraction. Let $TC(X)=k$ and
let $X\times X=U_1\cup\dots\cup U_k$ be an open cover together with continuous sections $s_i:U_i\to PX$.
We define sections $\sigma_i:U_i\cap (A\times A)\to PA$ by the formula $\sigma_i(a_1,a_2)(t)=r(s(a_1,a_2)(t))$.
\end{proof}

We recall that a family $\scr U$ of subsets of $X$ is called a {\em $k$-cover},
$k\in \N$ if every subfamily that consists of $k$ elements forms a cover of $X$.
We use the  following theorem ~\cite{Dr1}.
\begin{theorem}\label{criterion}
Let $\{U_0',\dots,U_n'\}$ be an open cover of a normal topological
space $X$. Then  for any $m=n,n+1,\dots,\infty$ there is an  open
$(n+1)$-cover of $X$, $\{U_k\}_{k=0}^{m}$ such that $U_k=U_k'$ for
$k\le n$ and $U_k=\cup_{i=0}^nV_i$ is a disjoint union with
$V_i\subset U_i$ for $k>n$.
\end{theorem}

\begin{cor}\label{deformable} 
Suppose that all sets $U_i'$, $i=0,\dots,n$, in the  theorem are (strictly) deformable in $X$ to a subspace $A\subset X$. Then the
sets $U_k$ for all $k$ are (strictly) deformable in $X$ to $A$. 
\end{cor}
The following proposition is well-known. The trick presented there can be traced back to the work of Kolmogorov on 13th Hilbert's problem~\cite{Os}.
\begin{prop}\label{KO} 
Let $U_0,\dots, U_{n+m}$ be  an $(n+1)$-cover of $X$ and let $V_0,\dots, V_{m+n}$ be  an $(m+1)$-cover of $Y$. 
Then the sets $W_k=U_k\times V_k$, $k=0,\dots n+m$, cover  $X\times Y$.
\end{prop}
\begin{proof} Let $(x,y)\in X\times Y$. A point $x$ is covered at least by $m+1$ elements. Otherwise $n+1$ elements that do not cover $x$
would not form a cover of $X$. That would give a contradiction with the assumption that  $U_0,\dots, U_{n+m}$ is an $(n+1)$-cover of $X$.
Let $x\in U_{i_0}\cap\dots\cap U_{i_m}$. By the assumption, the family $V_{i_0},\dots,V_{i_m}$ covers $Y$. Hence $y\in V_{i_s}$ for some $s$.
Then $(x,y)\in W_{i_s}$.
\end{proof}

\begin{thm}\label{wedge} For all ENR spaces $X$ and $Y$,
$$ \max\{TC(X), TC(Y), cat(X\times Y)\}\le TC(X\vee Y)\le$$
$$\le TC^M(X\vee Y)\le TC^M(X)+TC^M(Y)-1$$
\end{thm}
\begin{proof} 
Note that $TC(X\vee Y)\ge TC(X),TC(Y)$ by Proposition~\ref{retraction}.
Let $r_X:X\vee Y\to X$ and $r_Y:X\vee Y\to Y$ be the retraction collapsing the wedge  onto $X$ and $Y$
respectively.
The subset $$X\times Y\subset (X\vee Y)\times(X\vee Y)$$  is covered by $\le TC(X\vee Y)$ open sets  $U$ supplied with
a homotopy $$H_U:U\times I\to X\vee Y$$ such that $H(x,y,0)=x$ and $H(x,y,1)=y$. 
For each $U$ we define a homotopy $G:U\times I\to X\times Y$ by the formula
$$
G(x,y,t)=(r_XH_U(x,y,t),r_YH_U(x,y,1-t)).$$
Then $$G(x,y,0)=(r_XH_U(x,y,0),r_YH_U(x,y,1))=(r_X(x),r_Y(y))=(x,y)$$ and
$$G(x,y,1)=(r_XH_U(x,y,1),r_YH_U(x,y,0))=(r_X(y),r_Y(x))=(v_0,v_0)$$ where
$v_0$ is the wedge point in $X\vee Y$. Thus, $G$ contracts $U$ to a point in $X\times Y$.

Let $TC^M(X)=n+1$ and $TC^M(Y)=m+1$. Then there is an open cover $\tilde U_0,\dots,\tilde U_n$  of $X\times X$
with reserved sections $s_i:\tilde U_i\to PX$, $i=0,\dots,n$. Similarly, let $\tilde V_0,\dots,\tilde V_m$ be an open covering of $Y\times Y$ with 
reserved sections
$\sigma_j:\tilde V_j\to PY$, $j=0,\dots,m$. By Proposition~\ref{sections} all these sets are strictly deformable to the diagonal in $X\times X$ and
$Y\times Y$ respectively.
By Corollary~\ref{deformable} there is an open $(n+1)$-cover  $\tilde U_0,\dots,\tilde U_n,\dots,\tilde U_{n+m}$  of $X\times X$
by sets strictly deformable to the diagonal.
By Proposition~\ref{sections} there are strict deformations $$D_X^k:\tilde U_k\times I\to X\times X$$ of $\tilde U_k$ to $\Delta X$ that preserves faces $X\times v_0$ and $v_0\times X$. Similarly, 
there is an open $(m+1)$-cover $\tilde V_0,\dots,\tilde V_m,\dots,\tilde V_{m+n}$ of $Y\times Y$ and there are strict 
deformations $D_Y^k$ of $\tilde V_k$ in $Y\times Y$ to the diagonal $\Delta Y$ that preserves faces.

We use notations $$U_k=\tilde U_k\cap(X\times v_0)\ \ \text{and}\ \ V_k=\tilde V_k\cap(v_0\times Y),\ \ k=0,\dots,m+n.$$ 
Note that $U_0,\dots, U_{m+n}$ is an $(n+1)$-cover of $X\times v_0=X$ and $V_0,\dots, V_{m+n}$ is an
$(m+1)$-cover of $v_0\times Y=Y$. Let $W_k=U_k\times V_k$. By Proposition~\ref{KO} $W_0,\dots,W_{m+n}$ is an open cover
of $X\times Y$.

The deformations $D_X^k$ define the deformations $H_k:U_k\times I\to X\times v_0$ to the point $v_0\in X$
and the deformations $D_Y^k$ define the deformations $G_k:V_k\times I\to v_0\times Y$ to the point $v_0\in Y$. 
These deformations define the deformations
$$T_k:W_k\times I\to X\times Y$$ to the point $(v_0,v_0)$ such that if $W_k\cap (X\times v_0)\ne\emptyset$ then
$W_k\cap (X\times v_0)=U_k$ and  $T_k|_{U_k\times I}=H_k$ and if $W_k\cap(v_0\times Y)\ne\emptyset$ then
$W_k\cap(v_0\times Y)=V_k$ and $T_k|_{V_k\times I}=G_k$ for $k=0,\dots, m+n$.

Symmetrically, define $$U_k'=\tilde U_k\cap(v_0\times X)\ \ \text{and}\ \ \
V_k'=\tilde V_k\cap(Y\times v_0),\ \ \ k=0,\dots,m+n,$$ and corresponding deformations $$H_k':U_k'\times I\to X\ \ \text{and}\ \ G_k':V_k'\times I\to Y$$ to
the base points. Define $W'_k=U'_k\times V'_k$. By Proposition~\ref{KO}, the family $W_0',\dots, W_{n+m}'$ is an open cover of $Y\times X$. As before
there are deformations
$$T_k':W_k'\times I\to Y\times X$$ to the point $(v_0,v_0)$ such that if $W_k'\cap (v_0\times X)\ne\emptyset$, then
$W_k'\cap (v_0\times X)=U_k'$ and $T'_k|_{U'_k\times I}=H'_k$ and if $W_k'\cap (Y\times v_0)\ne\emptyset$, then
$W_k'\cap(Y\times v_0)=V_k$, $T'_k|_{V'_k\times I}=G'_k$ for $k=0,\dots, m+n$.

We  
define open sets
$$
O_k=W_k\cup W_k'\cup\tilde U_k\cup\tilde V_k\subset(X\vee Y)\times(X\vee Y),\ \ \  k=0,\dots,n+m$$
and note that  $\mathcal O=\{O_k\}$ covers $(X\vee Y)\times(X\vee Y)$. Note that the set
$$C=(X\vee Y)\times v_0\bigcup v_0\times(X\vee Y)$$ defines a partition of $(X\vee Y)\times(X\vee Y)$ in four pieces
$X\times X$, $X\times Y$, $Y\times X$, and $Y\times Y$. Also note that the intersection
$O_k\cap C\subset U_k\cup V_k\cup U_k'\cup V_k'$. By the construction  the deformations $D^k_X$, $D^k_Y$, $T_k$, and $T_k$ all
agrees on $O_k\cap C$. Therefore
the union of deformations
$$
T_k\cup T_k'\cup D_X^k\cup D_Y^k:O_k\times I\to (X\vee Y)\times(X\vee Y)
$$
is a well defined deformation $Q_k$ of $O_k$ to the diagonal $\Delta(X\vee Y)$.
Note that for all $k$, $Q_k$ are  strict deformations.
By Proposition~\ref{sections} each $Q_k$ defines a reserved section $\alpha_k:O_k\to P(X\vee Y)$.
Therefore, $$TX^M(X\vee Y)\le n+m+1=TC(X)+T(Y)-1.$$
\end{proof}

\begin{remark} A stronger version of the upper bound of Theorem~\ref{wedge} was proposed in
~\cite{F2}, (Theorem 19.1):
$$
 TC(X\vee Y)\le \max\{TC(X), TC(Y), \cat(X)+ \cat(Y)-1\}.
$$
Since the proof in ~\cite{F2} contains a gap, we call this inequality {\em Farber's Conjecture}. Note that Farber's inequality in view
of Theorem~\ref{wedge} would turns into the equality for spaces $X$ and $Y$
with $\Cat(X\times Y)=\Cat(X)+\Cat(Y)$.
\end{remark}

\begin{thm}\label{example}
(1)  There is a $2$-to-$1$ covering map $p:E\to B$ with $TC(E)>TC(B)$.

(2)  There is a finite complex $X$ with $TC(X)<TC(\tilde X)$  where $\tilde X$ is the universal covering of $X$.
\end{thm}
\begin{proof}
(1) We take $B=T\vee S^1$  where $T=S^1\times S^1$ is a 2-torus. Let $E$ to be the covering 
space defined by the 2-fold covering of
$S^1$. Note that $E$ is homeomorphic to the circle with two tori $T$ attached at antipodal points.
Thus, $E$ is homotopy equivalent to $T\vee T\vee S^1$. By Theorem~\ref{wedge} and Lemma~\ref{weinberger}
$$TC(B)\le TC^M(T)+TC^M(S^1)-1=\cat (T)+\cat (S^1)-1=3+2-1=4.$$ On the other hand  by Proposition~\ref{wedge},
$$TC(E)\ge \cat((T\vee S^1)\times T))=3+3-1=5.$$

(2) Consider $X=(S^3\times S^3)\vee S^1$. Since $S^3\times S^3$ is a connected Lie group,
by Lemma~\ref{weinberger},
$TC^M(S^3\times S^3)=cat(S^3\times S^3)=3$. By Theorem~\ref{wedge} 
$$TC(X)\le TC^M(S^3\times S^3)+ TC^M(S^1)-1=3+2-1=4.$$ Note that
the universal cover $\tilde X$ is homotopy equivalent to an infinite wedge $Y=\stackrel{\infty}\bigvee (S^3\times S^3)$. Then 
$Y$ admits a retraction onto
$(S^3\times S^3)\vee(S^3\times S^3)$.  By Proposition~\ref{retraction}, Theorem~\ref{wedge}, and the cup-length lower bound on $\cat$,
$$TC(\tilde X)\ge TC((S^3\times S^3)\vee(S^3\times S^3))\ge \cat(S^3\times S^3\times S^3\times S^3)\ge 5.$$
\end{proof}

\section{Topological complexity,  LS-category, and Schwartz genus }

We say a subset $A\subset X$ can be {\em  $rel\ \infty$ contracted to infinity} if for every compact subset $F\subset X$ there is a larger compact set $F\subset C$ and a homotopy $h_t:A\to X$ with $h_0=1_A$, $h_1(A)\cap F=\emptyset$ and $h_t(a)=a$ for $a\in A\setminus C$.
\begin{defin}
We define the $ rel\  \infty$ category $\infty\text{-}\cat(X)$ of a locally compact space $X$
as the minimal $k$ such that there is a cover $X=V_1\cup\dots\cup V_k$ by closed subsets where each $V_i$ can be $rel \ \ \infty$ contracted to infinity.
\end{defin}

\begin{remark}\label{one-point}  
It follows from the definition that for every locally compact space $X$, $$\cat(\alpha X)\le \infty\text{-}\cat(X)$$ where $\alpha X$ is the one-point compactification.
\end{remark}

\begin{question}\label{one-point}
Does the equality $\cat(\alpha X)= \infty\text{-}\cat(X)$ hold for all locally finite complexes with  tame ends?
\end{question}
We recall that $X$ has a tame end if there is a compactum $C\subset X$ such that $X\setminus \Int(C)\cong\partial C\times[0,1)$.

In the case when $\alpha X$ is a closed manifold this question could be related to the difference between the
category and the ball-category for manifolds. We recall that for a closed $n$-manifold $M$, $ballcat(M)\le k$ is there is a cover of $M$ by $k$ closed topological $n$-dimensional balls.
\begin{prop}
For any closed $n$-manifold $M$ and any $x_0\in M$, $$\cat (M)\le\infty\text{-}\cat(M\setminus\{x_0\})\le ballcat(M)\le \cat(M)+1.$$
\end{prop}
\begin{proof}
In view of Remark~\ref{one-point} and some known fact about the ball-category~\cite{CLOT}, only the second inequality needs a proof. Let $ballcat(M)=m$ and let $B_1,\dots, B_m$ be a cover of $M$ by
topological closed $n$-balls such that $x_0\notin \partial B_i$ for all $i$. Then all
$B_i\setminus\{x_0\}$ can be $rel\ \infty$ contracted in $M\setminus\{x_0\}$ to $x_0$.
\end{proof}
Since the one-point compactification of $X\times X$ with the diagonal $\Delta X$ removed is the quotient space $(X\times X)/\Delta X$, the following theorem shows that Question~\ref{one-point} is closely related to characterization of the topological complexity $TC^M$ by means of  the LS-category.
\begin{thm}\label{tc=ls} 
For any compact ENR $X$,
$$cat((X\times X)/\Delta X)\le TC^M(X)\le \infty\text{-}\cat((X\times X)\setminus\Delta X).$$
\end{thm}
\begin{proof}
Suppose that $TC^M(X)=k$.  Then  by the definition there is an open cover $U_1,\dots, U_k$ of $X\times X$ with  continuous  reserved sections
$s_i:U_i\to PX$ of $\pi:PX\to X\times X$. By Proposition~\ref{sections} there are strict deformations of $U_i$ in $X\times X$ to the diagonal $\Delta X$. They define the deformations of $U_i/(U_i\cap\Delta X)$ to the point $\{\Delta X\}$ in $(X\times X)/\Delta X$.
Thus, $\cat((X\times X)/\Delta X)\le k$.

Let $\infty\text{-}cat((X\times X)\setminus\Delta X)=k$ and let $(X\times X)\setminus\Delta X=F_1\cup\dots\cup F_k$ be the union of $k$ closed sets $rel\ \infty$
contractible to infinity. Let $W$ be a neighborhood of the diagonal $\Delta X$ in $X\times X$ that admits a deformation retraction $r_t$ to $\Delta X$.
Let $h_t^i$ be a deformation of $F_i$ into $W$. Then the concatenation of $h^i_t$ and $r_t$ defines a  deformation $H_i$ of $F_i$ to the diagonal. Let $\bar F_i=F_i\cup\Delta X$. Note that $H_i$ together with identity on $\Delta X$ define a strict deformation of $\bar F_i$ to the diagonal.
\end{proof}

\begin{remark} For the topological complexity $TC(X)$ a weaker version of the first inequality from Theorem~\ref{tc=ls} was proven in~\cite{F2}, Lemma 18.3.$$\cat((X\times X)/\Delta X)-1\le TC(X).$$  
\end{remark}

The topological complexity of $X$ equals the Schwarz genus of a certain fibration. It turns out
that for general fibrations we still have the inequalities  similar to
Theorem~\ref{tc=ls}.
\begin{thm}\label{schwartz}
For any fibration of compact spaces $p:X\to Y$, $$\cat (C_p)-1\le sg(p)\le \infty\text{-}\cat(C_p\setminus\{\ast\}).$$
\end{thm}
\begin{proof} We claim that
if a subset $U\subset Y$ admits a section $s:U\to X$, then $U$ is contractible in $C_p$. Indeed, it can be moved to $X$ in the mapping cylinder $M_p$.
Since the cone $\Con(X)$ is contained in $C_p$, it could be further contracted to a point.
Moreover, the mapping cylinder $\hat U=M_{p|_{p^{-1}(U)}}$ of the restriction of $p$ to the preimage $p^{-1}(U)$ is contractible in $C_p$, since it can be pushed to $U$ first. If $Y$ is covered by $n$ open sets $U_1,\dots, U_n$ each of which admits a section of $p$, then the mapping cylinder $M_p$ can pe covered by $n$ sets $\hat U_1,\dots,\hat U_n$ all contractible in the mapping cylinder $C_p$. Since $C_p=M_p\cup \Con(X)$, the open enlargements of the sets  $\hat U_1,\dots,\hat U_n$, and $\Con(X)$ define an open cover of $C_p$ by $n+1$ elements all contractible in $C_p$. Hence $\cat(C_p)-1\le sg(p)$.

Suppose that $\infty\text{-}\cat(C_p\setminus\{\ast\})\le n$. Let $V_1,\dots V_n$ be a closed cover of $C_p\setminus\{\ast\}$
by sets that can be $rel\ \infty$ contracted to infinity. Let $$H_i:V_i\times I\to C_p\setminus\{\ast\}$$ be a contraction such that $$H_i(V_i\times 1)\subset \Con(X)\setminus\{\ast\}\subset C_p\setminus\{\ast\}.$$ We define $F_i=V_i\cap Y\subset C_p$. Let $\pi:\Con(X)\setminus\{\ast\}\to X$ be the projection.
By the Homotopy Lifting Property, the homotopy $p\circ H_i|_{F_i\times [0,1]}:F_i\times [0,1]\to Y$ has a lift $H_i':F_i\times[0,1]\to X$ which coincides with $\pi\circ H_i$ on $F_i\times 1$. Then $H_i'$ restricted to $F_i\times 0$ is a section of $p$ over $F_i$. Thus, $sg(p)\le \infty\text{-}\cat(C_p\setminus\{\ast\}).$
\end{proof}

The following example shows that neither of the two inequalities of Theorem~\ref{schwartz} can be improved.

\begin{ex} 
(1) For the identity map $1_X:X\to X$ in view of the equality $C_{1_X}=\Con(X)$ we obtain: $$\cat (C_{1_X})-1=0<sg(1_X)=1=\cat(Con(X))=\infty\text{-}\cat(C_{1_X}\setminus\{\ast\}).$$

For  the square map $p:S^1\to S^1$, $p(z)=z^2$,
$$\cat(C_p)-1=2=sg(p)<3=\cat(C_p)\le \infty\text{-}\cat(C_p\setminus\{\ast\}),$$ since $C_p=\R P^2$ and $\cat(\R P^2)=3.$
\end{ex}

\section{On the Arnold-Kuiper theorem}

\begin{thm}\label{cat of orbit} The non-reduced Lusternik-Schnirelmann category of the orbit space $\C P^2/\Z_2$ of the action of $\Z_2$ on the complex projective plane $\C P^2$ by the conjugation is $2$,
$$cat(\C P^2/\Z_2)=2.$$
\end{thm}

\begin{cor}[Arnold, Kuiper]
The orbit space $\C P^2/\Z_2$ of the action of $\Z_2$ on the complex projective plane $\C P^2$ by the conjugation is a $4$-sphere.
\end{cor}
\begin{proof} Clearly, the fixed point set of this action is a real projective plane 
$$\R P^2=\{[a:b:c]\mid a,b,c\in\R, |a|+|b|+|c|\ne 0\}\subset$$
$$\subset\{[a:b:c]\mid a,b,c\in\C, |a|+|b|+|c|\ne 0\}=\C P^2.$$
Moreover, the action preserves the normal bundle to $\R P^2$.
Therefore, the orbit space $\C P^2/\Z_2$ is a 4-manifold. A closed $n$-manifold of the category 2 is homotopy equivalent to the $n$-sphere (see~\cite{CLOT}).
Then by Freedman's theorem~\cite{Fr}, $\C P^2/\Z_2$   is homeomorphic to the 4-sphere.
\end{proof}

\begin{remark}
We note that Arnold and Kuiper proved a diffeomorphism theorem. Since the smooth 4-dimensional Poincare conjecture is still a conjecture,
here we can provide only a homeomorphism.
\end{remark}

We identify the 2-sphere $S^2$ with the one-point compactification $\C\cup\infty$ of the complex plane. Then $\Z_2$-action on $\C$ by the
 conjugation extends to an action on $S^2$. Clearly, a $\Z_2$-action on $S^2$ extends to an action on the symmetric $n$th power $SP^n(S^2)$ of $S^2$. We recall that $SP^nX=X^n/\Sigma_n$ is the orbit space
 on the $n$th power $X^n$ under the action of the symmetric group $\Sigma_n$ by  permutation of coordinates.

\begin{prop}\label{known}
There is a $\Z_2$-equivariant homeomorphism between complex projective space $\C P^2$ and the symmetric square $SP^2(S^2)$.
\end{prop}
\begin{proof}
The points  $[a:b:c]\in\C P^2$ are in bijection with non-degenerate quadratics $ax^2+bxy+cy^2$. Any factorization of this quadratic
 $$ax^2+bxy+cy^2=(a_1x+b_1y)(a_2x+b_2y)$$ defines the same non-ordered (perhaps repeated) pairs of points 
$$\frac{a_1}{b_1},\frac{a_2}{b_2}\in \C\cup\infty=S^2.$$
Note that the non-degeneration condition $|a|+|b|+|c|\ne 0$ implies that $a_i$ and $b_i$  cannot be all equal zero for $i=1,2$.
Also we use  the standard convention $\frac{z}{0}=\infty$ for any $z\in\C$.

This correspondence is the required homeomorphism.
\end{proof}

\begin{remark} The above proposition is an equivariant version of the well-known fact: $\C P^n\cong SP^n(S^2)$.
\end{remark}

\

{\em Proof of Theorem~\ref{cat of orbit}.}
We present $M=SP^2(S^2)/\Z_2=F\cup U$ as a union of two contractible sets one closed and one open.
Note that the set $U=SP^2(\C)/\Z_2$ is open and contractible, since $\C$ is contractible to a point equivariantly.
The equator $S^1=\R\cup\infty\subset S^2$ separates $S^2$ in two hemispheres $D_-$ and $D_+$. We show that the complement $F=M\setminus U$
admits a continuous bijection onto the closed upper hemisphere $\bar D_+$. Indeed, it consists of non-ordered pairs of pairs
$\{\infty,z\}$, $\{\infty,\bar z\}$ where $z\in \bar D_+$. This defines the bijection which is clearly continuous. Since $F$ is compact, it is homeomorphic to $\bar D_+$
and hence is contractible. Since $F$ is an absolute retract and $M$ is absolute neighborhood retract, there is an open neighborhood $V$ of $F$ in $M$
that contracts to $F$ in $M$ and, hence, to a point. Thus, $M$ is covered by two open sets $U$ and $V$, both contractible in $M$. \qed

\end{document}